\documentclass [reqno,10pt,oneside]{amsart}

\pdfoutput=1
\usepackage{hyperref}
\usepackage{amsthm}
\usepackage{amscd}
\usepackage{amsmath,amssymb}
\usepackage{dsfont}
\usepackage{algorithm}
\usepackage[noend]{algorithmic}

\usepackage{slashbox}

\theoremstyle {plain}
\newtheorem {thm}{Theorem}[section]
\newtheorem {prop}[thm]{Proposition}

\theoremstyle {definition}
\newtheorem {defn}[thm]{Definition}
\theoremstyle {remark}
\newtheorem {rem}[thm]{Remark}
\newtheorem {nota}[thm]{Notation}

\newtheorem {exmp}[thm]{Example}
\newtheorem {conv}[thm]{Convention}

\DeclareMathOperator{\Aut}{Aut}

\DeclareMathOperator{\Mon}{Mon}
\DeclareMathOperator{\LM}{LM}
\DeclareMathOperator{\LC}{LC}
\DeclareMathOperator{\LT}{LT}
\DeclareMathOperator{\tail}{tail}
\DeclareMathOperator{\id}{id}

\DeclareMathOperator{\sign}{sign}
\DeclareMathOperator{\ord}{ord}
\DeclareMathOperator{\ch}{char}

\DeclareMathOperator{\Sym}{Sym}
\DeclareMathOperator{\Mat}{Mat}
\DeclareMathOperator{\lcm}{lcm}

\DeclareMathOperator{\spoly}{spoly}
\DeclareMathOperator{\NF}{NF}

\newcommand{\A}{{\mathbb A}}
\newcommand{\C}{{\mathbb C}}

\newcommand{\F}{{\mathbb F}}
\newcommand{\N}{{\mathbb N}}
\newcommand{\Q}{{\mathbb Q}}

\newcommand{\sym}{{\mathds S}}
\newcommand{\gen}[1]{\left\langle #1 \right\rangle}
\newcommand{\set}[2]{\left\{{#1},\ldots,{#2}\right\}}
\newcommand{\singular}{{\sc Singular }}
\renewcommand{\matrix}[2]{\left(\begin{array}{#1}#2\end{array}\right)}

\begin{document}

\bibliographystyle{alpha}

\title{Gr\"obner bases of symmetric ideals}

\author{Stefan Steidel}
\address{Stefan Steidel\\ Department of Mathematical Methods in Dynamics and Durability\\
 Fraunhofer Institute for Industrial Mathematics ITWM\\ 
 Fraunhofer-Platz 1\\ 67663 Kaiserslautern\\ Germany}
\email{stefan.steidel@itwm.fraunhofer.de} 


\date{\today}

\maketitle

\begin{abstract}
In this article we present two new algorithms to compute the Gr\"obner basis of an ideal that is invariant 
under certain permutations of the ring variables and which are both implemented in \textsc{Singular} 
(cf.\ \cite{DGPS12}). 
The first and major algorithm is most performant over finite fields whereas the second algorithm is a 
probabilistic modification of the modular computation of Gr\"obner bases based on the articles by Arnold
(cf.\ \cite{A03}), Idrees, Pfister, Steidel (cf.\ \cite{IPS11}) and Noro, Yokoyama (cf.\ \cite{NY12}, \cite{Y12}). 
In fact, the first algorithm that mainly uses the given symmetry, improves the necessary modular calculations 
in positive characteristic in the second algorithm. Particularly, we could, for the first time even though 
probabilistic, compute the Gr\"obner basis of the famous ideal of \emph{cyclic $9$-roots} (cf.\ \cite{BF91}) 
over the rationals with \textsc{Singular}. 
\end{abstract}

%\begin{keyword}
%Gr\"obner bases, symmetry, modular computation, parallel computation
%\end{keyword}

\section{Introduction}
 
Computing the Gr\"obner basis of an ideal is a powerful tool in commutative algebra, with applications in 
algebraic geometry and singularity theory. The first general algorithm was proposed by Buchberger in 
1965 (cf.\ \cite{Bu65}). 

There are previous works by Aschenbrenner and Hillar on symmetric Gr\"obner bases in infinite-dimensional
rings (cf.\ \cite{AH07}, \cite{AH08}) and Faug\`ere and Rahmany using SAGBI-Gr\"obner bases for solving systems of
polynomial equations with symmetries (cf.\ \cite{FR09}).

Within this article we improve the computation of Gr\"obner bases in case that the input ideal has some 
special symmetry-character.
Consider, for example, the ideal $I=\gen{x^2y^2 - z,xy - 2y + 3z,xy - 2x + 3z} \subseteq \Q[x,y,z]$. 
Then $I$ does not vary if one interchanges the variables $x$ and $y$, and we say that $I$ is 
\emph{symmetric} with respect to the permutation $x \longleftrightarrow y$. 
In the following we use this property to manipulate the ideal by an appropriate linear transformation and 
apply Buchberger's algorithm subsequently.

We start in Section \ref{secNotDef} by presenting some basic notations and definitions. 
In Section \ref{secSymmGB} we introduce the symmetric Gr\"obner basis algorithm, and state a theoretical 
result that justifies the impact of the symmetry in Proposition \ref{propPropertyGB}.  
Section \ref{secSymmGBQ} combines the symmetric algorithm of Section \ref{secSymmGB} with modular 
methods which results in an probabilistic but performant algorithm over the rationals.
Moreover, examples and timings are provided in Section \ref{subsecSymmGBExTime} and Section 
\ref{subsecSymmGBQExTime}, respectively.

\section{Basic notations and definitions} \label{secNotDef}

Let $\sigma \in \sym_n := \Sym(\set 1n)$ be a permutation. The \emph{order} of $\sigma$ is the minimal 
natural number $k \in \N_{>0}$ such that $\sigma^k = \id$, in particular $\ord(\sigma):=\#(\gen \sigma) < 
\infty$.
In order to describe $\sigma$ properly, we make use of the following well-known result concerning the 
representation of permutations.

\begin{defn} \label{defCycDecomp}
Let $\sigma \in \sym_n$ be a permutation. Then there exists a natural number $\vartheta(\sigma)$ and 
a finite disjoint partition $\set 1n = \coprod_{i=1}^{\vartheta(\sigma)} \set{e_{i,1}}{e_{i,l_i}}$ such that 
$$\sigma = (e_{1,1} \ldots e_{1,l_1})\cdots (e_{\vartheta(\sigma),1}\ldots e_{\vartheta(\sigma),
l_{\vartheta(\sigma)}})$$ with $l_1 + \ldots + l_{\vartheta(\sigma)} = n$ and $0 \leq l_i \leq n$ for all 
$i \in \set{1}{\vartheta(\sigma)}$. 
The cycles $(e_{i,1} \ldots e_{i,l_i})$ are up to alignment uniquely determined, and we call this 
representation the \emph{cycle decomposition} of $\sigma$.
The tuple $(l_1,\ldots,l_{\vartheta(\sigma)})$ is called the \emph{cycle type} of $\sigma$ if $l_1 \leq 
\ldots \leq l_{\vartheta(\sigma)}$. 
\end{defn}

Note that having the cycle decomposition of a permutation $\sigma$ it holds $\ord(\sigma) = 
\lcm(l_1,\ldots,l_{\vartheta(\sigma)})$. From now on we assume that all considered permutations 
$\sigma \in \sym_n$ are given in cycle decomposition.

%\begin{exmp}
%Let $\sigma = (135)(2)(46)(7) \in \sym_7$ be the cycle decomposition of 
%the permutation $\sigma$. Then $\ord(\sigma)=\lcm(3,1,2,1)=6$, and $\sigma$ has cycle type
%$(3,2,1,1)$.
%\end{exmp}

Now let $K$ be a field and $X := \set{x_1}{x_n}$ be a set of indeterminates, 
then $\sigma \in \sym_n$ induces a canonical automorphism on the polynomial ring over $K$ in
these indeterminates, $K[X]$, via $\varphi_\sigma: K[X] \longrightarrow K[X], \, x_i \longmapsto      
x_{\sigma(i)}$.
By abuse of notation we always write $\sigma$ instead of $\varphi_\sigma$, i.e.\ we identify 
the group $\sym_n$ as a subgroup of the automorphism group $\Aut(K[X])$.

\begin{defn}
Let $I \subseteq K[X]$ be an ideal and $\sigma \in \Aut(K[X])$ be an automorphism.
Then $I$ is called \emph{$\sigma$-symmetric} if $\sigma(I)=I$. Moreover, let $\mathcal S \subseteq 
\Aut(K[X])$ be a subgroup then we call $I$ \emph{$\mathcal S$-symmetric} if it is $\sigma$-symmetric 
for all $\sigma \in \mathcal S$. 
\end{defn}

Every subgroup of $\sym_n$ has only finitely many elements and is therefore finitely generated. 
Hence, let $\mathcal S = \gen{\sigma_1,\ldots,\sigma_m} \subseteq \sym_n$ then an ideal $I \subseteq 
K[X]$ is $\mathcal S$-symmetric if and only if it is $\sigma_i$-symmetric for all $i \in \set 1m$. In particular, 
if an ideal is $\sigma$-symmetric then it is $\gen \sigma$-symmetric.

%\begin{proof}
%Let $f \in I$ and $\sigma \in \mathcal S$ be arbitrary. By assumption there exists a natural 
%number $m \in \N$ and $\delta_1,\ldots, \delta_t \in \{\sigma_1,\ldots,\sigma_m\}$ such that 
%$\sigma = \prod_{j=1}^t \delta_j$.
%Thus $\sigma(f) = \delta_1(\ldots\delta_t(f)) \in I$ since $I$ is $\sigma_i$-symmetric for every 
%$i \in \set 1m$.
%\end{proof}

Moreover, given an ideal $I \subseteq K[X]$ we can always choose a finite set of polynomials $F_I =
\{f_1,\ldots,f_r\}$ such that $I = \gen{F_I}$.
Thus, if $I$ is $\sigma$-symmetric with $\sigma \in \Aut(K[X])$ we even may assume that $\sigma(F_I)
=F_I$ by possibly adding some polynomials to $F_I$.

\begin{exmp} \label{exmpMain0}
The ideal $I = \gen{x^2y^2 - z,\; xy - 2y + 3z,\; xy - 2x + 3z} \subseteq \Q[x,y,z]$ is obviously 
$\sigma$-symmetric for $\sigma = (12)(3) \in \sym_3$.
\end{exmp}

We denote by $\Mon(X)$ the set of monomials. Moreover, if $>$ is a monomial ordering and $f \in K[X]$ 
a polynomial, then we denote by $\LC(f)$ the leading coefficient of $f$, by $\LM(f)$ the leading monomial 
of $f$, by $\LT(f)$ the leading term (leading monomial with leading coefficient) of $f$, and by $\tail(f) = 
f-\LT(f)$ the tail of $f$ with respect to the ordering $>$. In particular, with our notation it holds $\LT(f) = \LC(f) 
\cdot \LM(f)$.

\begin{conv}
In the following $>$ is a degree ordering, and we always consider reduced Gr\"obner bases $G$,
that is $0 \notin G$, $\LM(g) \nmid \LM(f)$ for any two elements $f \neq g$ in $G$, and $\LC(g) = 1$ 
respectively no monomial of $\tail(g)$ is contained in the leading ideal of $G$ for any $g \in G$.
\end{conv}

\section{Gr\"obner bases using symmetry} \label{secSymmGB}

Within this section we describe how to achieve an improvement of the Gr\"obner basis computation 
of a $\sigma$-symmetric ideal $I \subseteq K[X]$ by using its symmetric property. The basic idea is the 
construction and usage of an appropriate linear transformation $\tau \in \Aut(K[X])$ which ``diagonalises''
$\sigma$ and respects the $\sigma$-symmetry of $I$. 
It turns out that in many cases the usual Gr\"obner basis computation on the 
transformed side is much faster than the computation on the original side. 
Since the pull back of this Gr\"obner basis is in general not a Gr\"obner basis anymore we have to add 
another Gr\"obner basis computation. Nevertheless, this indirection effects an enormous speed up 
compared to the usual Gr\"obner basis algorithm (cf.\ Section \ref{subsecSymmGBExTime}).

We assume that the tuple $(\sigma,K)$ with $\ord(\sigma) = k \in \N$ always 
satisfies $\ch(K) \nmid k$ and $K$ has a $k$-th primitive root of unity $\xi_k$.

\begin{rem} \label{remFieldExtension}
We can always achieve this assumption by possibly adjoining $\xi_k$. In particular, we can swap to $K[\xi_k]$ 
by working over the field $K[a]/\Phi_k(a)$ where $\Phi_k(a)$ is the $k$-th cyclotomic polynomial.
\end{rem}

\subsection{The symmetric Gr\"obner basis algorithm} \label{subsecSymmGBAlg}

We start by illuminating the basis for the symmetric Gr\"obner basis algorithm from a character theoretical point 
of view and in terms of linear algebra. 

Therefore, we consider the $n$-dimensional $K$-subvector space $V = \gen{x_1,\ldots,x_n}_K$ of the 
infinite-dimensional $K$-vector space $K[X]$. Then due to our assumption that $\ch(K) \nmid \#(\gen \sigma)$, 
character theory guarantees that every representation of $\gen \sigma \subseteq \sym_n$ is a direct sum of 
irreducible representations (cf.\ \cite[Theorem 2]{Se96}), and all irreducible representations of $\gen \sigma 
\subseteq \sym_n$ have degree $1$ since $\gen \sigma \subseteq \sym_n$ is an abelian group 
(cf.\ \cite[Theorem 9]{Se96}). In particular, the representation $\rho: \gen \sigma \longrightarrow \Aut(V)$ is 
diagonalisable, i.e.\ $V = \bigoplus_{i=1}^n V_i$ with $V_i = \gen{y_i}_K$ and $\rho(\sigma)(y_i) = \xi_k^{\nu_i} 
\cdot y_i$ for some $0 \leq \nu_i \leq k-1$.

In terms of linear algebra we have the following quite simple proposition which, together with its proof, forms 
the basis of the symmetric Gr\"obner basis algorithm.

\begin{prop} \label{propAut}
Let $\sigma \in \sym_n$ have cycle type $(l_1,\ldots,l_{\vartheta(\sigma)})$. Then $\sigma \in \Aut(V)$ is 
diagonalisable with eigenvalues $\{\xi_{k}^{kj/l_m} \mid 1 \leq m \leq \vartheta(\sigma),\, 0 \leq j \leq l_m-1\}$. 
\end{prop}

\begin{proof}
Let $\sigma = (e_1 \ldots e_n) \in \sym_n$ with $\{e_1,\ldots,e_n\} = \{1,\ldots,n\}$ be an $n$-cycle.
The columns of the representation matrix $M(\sigma,X)$ of $\sigma \in \Aut(V)$ with respect to the 
$K$-basis $X=\{x_1,\ldots,x_n\}$ of $V$ are just the permuted unit vectors of $K^n$. Hence, $M(\sigma,X)$ 
is a unitarian matrix and in particular diagonalisable.
Moreover, let $C = t\mathds 1_n-M(\sigma,X) \in \Mat(n\times n,K[t])$ then the characteristic polynomial 
of $\sigma \in \Aut(V)$ is 
\begin{align*}
\chi_\sigma = \det(C) & = \sum_{\pi \in \sym_n} \sign(\pi) \cdot c_{1\pi(1)} \cdots c_{n\pi(n)} \\ & = t^n + 
\sign(\sigma) \cdot (-1)^n = t^n +(-1)^{n-1} \cdot (-1)^n = t^n -1,
\end{align*}
and $\{1,\xi_n,\xi_n^2,\ldots,\xi_n^{n-1}\}$ are exactly the eigenvalues of $\sigma$. Now, consider the 
combinatorial set $$Y := \left\{y_{e_i} = \sum_{j=1}^n \xi_n^{(i-1)(j-1)} \cdot x_{\sigma^{j-1}(e_i)} \;\Bigg|\; 
1 \leq i \leq n\right\}.$$ 
Note that $Y$ is a $K$-basis of $V$ since the coefficients of each $y_{e_i}$ are just different powers of 
the primitive root of unity $\xi_n$, and $\{x_{\sigma^{j-1}(e_i)} \mid 1 \leq j \leq n\} = X$ is a $K$-basis of 
$V$ itself.

Let $i \in \{1,\ldots,n\}$, then we easily compute that $\sigma(y_{e_i}) = \xi_n^{i-1} \cdot y_{e_i}$.
Consequently, $y_i$ is the eigenvector corresponding to the eigenvalue $\xi_n^{i-1}$, and the representation 
matrix $M(\sigma,Y)$ of $\sigma \in \Aut(V)$ with respect to $Y$ is diagonal.

Now, let $\sigma \in \sym_n$ have cycle type $(l_1,\ldots,l_{\vartheta(\sigma)})$ with cycle decomposition $\sigma = 
\sigma_1 \cdots \sigma_{\vartheta(\sigma)}$ and $\sigma_m = (e_{m,1} \ldots e_{m,l_m})$ for $1\leq m \leq \vartheta(\sigma)$.
Then we have $\ord(\sigma_m) = l_m$, $\ord(\sigma) = k = \lcm(l_1,\ldots,l_{\vartheta(\sigma)})$, and $$\xi_{l_m} = 
\xi_k^{k/l_m} \in K$$ is an $l_m$-th primitive root of unity. We set $X_m = \{x_{e_{m,1}},\ldots,x_{e_{m,l_m}}\}$ such that 
$\sigma_m \in \Aut(V_m)$ with $V_m = \gen{X_m}_K$, and $X = X_1 \cup \ldots \cup X_{\vartheta(\sigma)}$ is a $K$-basis
of $V = V_1 \oplus \ldots \oplus V_{\vartheta(\sigma)}$. 
Analogously to the $n$-cycle case we obtain the combinatorial sets $$Y_m := \left\{ y_{e_{m,i}} = \sum_{j=1}^{l_m} 
\xi_{l_m}^{(i-1)(j-1)} \cdot x_{\sigma_m^{j-1}(e_{m,1})} \;\Bigg|\; 1 \leq i \leq l_m\right\}$$ of eigenvectors of $\sigma_m \in 
\Aut(V_m)$ so that the representation matrix $M(\sigma_m,Y_m)$ of $\sigma_m \in \Aut(V_m)$ with respect to the 
$K$-basis $Y_m$ of $V_m$ is diagonal with eigenvalues \linebreak $\{1,\xi_{l_m},\xi_{l_m}^2,\ldots, \xi_{l_m}^{l_m-1}\}$. 
Hence, $Y = Y_1 \cup \ldots \cup Y_{\vartheta(\sigma)}$ is a $K$-basis of $V$, consists of eigenvectors of $\sigma \in 
\Aut(V)$, and the representation matrix 
$$M(\sigma,Y) = \left( \begin{array}{ccc}
         M(\sigma_1,Y_1) &        & \\
                             & \ddots & \\
                             &        & M(\sigma_{\vartheta(\sigma)},Y_{\vartheta(\sigma)})
         \end{array} \right) \in \Mat(n \times n, K)$$
of $\sigma \in \Aut(V)$ with respect to $Y$ is diagonal with eigenvalues $\{\xi_{l_m}^j \mid 1 \leq m \leq 
\vartheta(\sigma),\, 0 \leq j \leq l_m-1\} = \{\xi_{k}^{kj/l_m} \mid 1 \leq m \leq \vartheta(\sigma),\, 0 \leq j \leq 
l_m-1\}$.
\end{proof}

\begin{rem} \label{remAut}
Due to the constructive proof of Proposition \ref{propAut} the eigenvectors of $\sigma \in \Aut(V)$ can 
be obtained by purely combinatorial methods which is quite profitable from the algorithmic point of view. 
Hence, let us define the ring homomorphism 
\begin{align*}
\tau: K[X] & \longrightarrow  K[X] \\
           x_{e_{m,i}} & \longmapsto   y_{e_{m,i}}=\sum_{j=1}^{l_m} \xi_{l_m}^{(i-1)(j-1)} \cdot x_{\sigma_m^{j-1}(e_{m,i})}
\end{align*}
which maps the ring-variables onto the eigenvectors of $\sigma \in \Aut(V)$. Consequently, we can define 
another ring homomorphism $\sigma_\tau$ induced by the ring automorphism $\sigma \in \Aut(K[X])$, the 
linear transformation $\tau \in \Aut(K[X])$ and the commutative diagram
$$\begin{CD}
  K[X]       @>\sigma>>       K[X] \\
  @VV \tau V                  @VV \tau V \\
  K[X]       @>\sigma_\tau>>  K[X]
\end{CD}$$
so that $\sigma_\tau = \tau\sigma\tau^{-1}$ satisfies the property that $\sigma_\tau(x_i) = \xi_k^{\nu_i} \cdot 
x_i$ for suitable exponents $0 \leq \nu_i \leq k-1$ and all $1 \leq i \leq n$.
\end{rem}
 
\begin{exmp} \label{exmpMain1}
Let $\sigma = (12)(3) \in \sym_3$ with $\ord(\sigma) = 2$. Then consider $\xi_2 = -1 \in \Q$ and construct
\begin{align*}
\tau: \Q[x,y,z] & \longrightarrow \Q[x,y,z] \\ 
                     x & \longmapsto (-1)^{0 \cdot 0} \cdot x + (-1)^{0 \cdot 1} \cdot y = x + y, \\
                     y & \longmapsto (-1)^{1 \cdot 0} \cdot y + (-1)^{1 \cdot 1} \cdot x = y - x, \\
                     z & \longmapsto z
\end{align*}
as in Remark \ref{remAut}.
Hence, $\tau$ is bijective with inverse $\tau^{-1}$ defined by 
$$\tau^{-1}(x) = \frac{x - y}{2}, \quad \tau^{-1}(y) = \frac{x + y}{2}, \quad \tau^{-1}(z) = z.$$
Referring to Remark \ref{remAut} $\sigma_\tau$ is induced by $\sigma_\tau= \tau \sigma \tau^{-1}$ 
and thus it holds
\begin{align*}
\sigma_\tau(x) & = \tau\sigma\tau^{-1}(x) = \tau\sigma\Big(\frac{x-y}{2}\Big) = \tau\Big(\frac{y-x}{2}\Big) = 
\frac{y-x-x+y}{2} = -x, \\
\sigma_\tau(y) & = \tau\sigma\tau^{-1}(y) = \tau\sigma\Big(\frac{x+y}{2}\Big) = \tau\Big(\frac{x+y}{2}\Big) = 
\frac{x+y+(y-x)}{2} = y, \\
\sigma_\tau(z) & = \tau\sigma\tau^{-1}(y) = \tau\sigma(z) = \tau(z) =z.
\end{align*}
\end{exmp}

\begin{nota}
For a better understanding we will index objects that live on the transformed side by $\tau$.
\end{nota}

As aforementioned respectively proven above, the induced automorphism $\sigma_\tau$ has a nice 
multiplication property on the ring variables that, however, is a priori not a sufficient reason for a fast 
Gr\"obner basis computation. But, in addition, the linear transformation $\tau$ also respects the 
symmetry of the input ideal.

\begin{prop} \label{propSymmAut}
If the ideal $I \subseteq K[X]$ is $\sigma$-symmetric, then the transformed ideal $I_\tau := \tau(I) \in K[X]$ 
is $\sigma_\tau$-symmetric. 
\end{prop}

\begin{proof} 
Let $I = \gen{f_1,\ldots,f_r}$. By definition of $\sigma_\tau$ we obtain for $\sigma(f_i) = f_j$ that 
$\sigma_\tau(\tau(f_i)) = \tau(\sigma(f_i)) = \tau(f_j)$. Thus, the ideal $I_\tau := \tau(I) = \gen{\tau(f_1),\ldots,
\tau(f_r)}$ is $\sigma_\tau$-symmetric. 
\end{proof}

\begin{exmp} \label{exmpMain2}
Let $>=>_{dp}$ be the degree reverse lexicographical ordering\footnote{\emph{Degree reverse lexicographical 
ordering:} Let $X^\alpha, X^\beta \in \Mon(X)$. $X^\alpha >_{dp} X^\beta \, :\Longleftrightarrow \, \deg(X^\alpha) 
> \deg(X^\beta)$ or $(\deg(X^\alpha) = \deg(X^\beta)$ and $\exists \, 1 \leq i \leq n: \; \alpha_n = \beta_n, 
\ldots, \alpha_{i-1} = \beta_{i-1}, \alpha_i < \beta_i)$, where $\deg(X^\alpha) = \alpha_1 + \ldots + \alpha_n$; 
cf.\ \cite{GP07}.}, and $\sigma,\tau,\sigma_\tau$ as in Example \ref{exmpMain1}.
Now we consider the $\sigma$-symmetric ideal $$I = \gen{x^2y^2 - z,\; xy - 2y + 3z,\; xy - 2x + 3z} 
\subseteq \Q[x,y,z]$$ and obtain that the transformed ideal
\begin{align*}
I_\tau := \tau(I) = \big\langle x^4-2x^2y^2+y^4-z,\, & -x^2+y^2+2x-2y+3z,  \\
                                              & -x^2+y^2-2x-2y+3z \big\rangle \subseteq \Q[x,y,z]
\end{align*}
is $\sigma_\tau$-symmetric.
\end{exmp}

Due to Proposition \ref{propAut} and Proposition \ref{propSymmAut} we see that transforming the original 
ideal via $\tau$ still respects some symmetry. 
In particular, the ideal $I_\tau$ is $\sigma_\tau$-symmetric and applying the automorphism $\sigma_\tau$ 
on any variable, respectively monomial, effects just a multiplication by a power of a primitive root of unity. 
The advantage of this circumstance is the fact that the symmetry propagates during the process of computing 
a Gr\"obner basis of $I_\tau$ which influences the performance in a positive way. More precisely, the following 
proposition holds.

\begin{prop} \label{propPropertyGB}
Let $I_\tau$ be $\sigma_\tau$-symmetric, then a Gr\"obner basis $G_\tau$ of $I_\tau$ satisfies 
$\sigma_\tau(g_\tau) = \xi_k^{\nu_{g_\tau}} \cdot g_\tau$ for all $g_\tau \in G_\tau$ with suitable 
$0 \leq \nu_{g_\tau} \leq k-1$.
\end{prop}

\begin{proof}
Let $I_\tau = \gen{F_{I_\tau}}$ and $f,g \in F_{I_\tau}$. Due to the property of $\sigma_\tau$ we can define 
$X^\alpha :=\LM(f)=\LM(\sigma_\tau(f))$, $X^\beta:=\LM(g)=\LM(\sigma_\tau(g))$, $X^\gamma:=
\lcm(X^\alpha,X^\beta)$ and $\sigma_\tau(X^\alpha)=\xi_k^{\nu_\alpha} X^\alpha$, $\sigma_\tau(X^\beta)
=\xi_k^{\nu_\beta} X^\beta$, $\sigma_\tau(X^\gamma) =\xi_k^{\nu_\gamma} X^\gamma$ for suitable 
$\nu_\alpha, \nu_\beta, \nu_\gamma \in \set{0}{k-1}$.
Then $$\spoly(f,g) = X^{\gamma-\alpha}\cdot f - \frac{\LC(f)}{\LC(g)} \cdot X^{\gamma-\beta} \cdot g$$
and again by the property of $\sigma_\tau$ it holds $\LC(\sigma_\tau(f)) = \xi_k^{\nu_\alpha} \cdot \LC(f)$ 
respectively $\LC(\sigma_\tau(g)) = \xi_k^{\nu_\beta} \cdot \LC(g)$. Thus
\begin{align*}
\sigma_\tau(\spoly(f,g)) & = \sigma_\tau(X^{\gamma-\alpha}) \cdot \sigma_\tau(f) - \frac{\LC(f)}{\LC(g)} \cdot 
                                                  \sigma_\tau(X^{\gamma-\beta}) \cdot \sigma_\tau(g) \\
                                           & = \xi_k^{\nu_\gamma - \nu_\alpha} \cdot X^{\gamma - \alpha} \cdot \sigma_\tau(f)
                                                  - \frac{\LC(f)}{\LC(g)} \cdot \xi_k^{\nu_\gamma - \nu_\beta} \cdot X^{\gamma - \beta}
                                                  \cdot \sigma_\tau(g) \\
                                           & = \xi_k^{\nu_\gamma - \nu_\alpha} \cdot \left( X^{\gamma-\alpha} \cdot \sigma_\tau(f) 
                                                  - \frac{\LC(f)}{\LC(g)} \cdot \xi_k^{\nu_\alpha - \nu_\beta} \cdot X^{\gamma - \beta}
                                                  \cdot \sigma_\tau(g) \right) \\
                                           & = \xi_k^{\nu_\gamma - \nu_\alpha} \cdot \left( X^{\gamma-\alpha} \cdot \sigma_\tau(f) 
                                                  - \frac{\LC(\sigma_\tau(f))}{\LC(\sigma_\tau(g))} \cdot X^{\gamma - \beta}
                                                  \cdot \sigma_\tau(g) \right) \\
                                           & = \xi_k^{\nu_\gamma-\nu_\alpha} \cdot \spoly(\sigma_\tau(f),\sigma_\tau(g)).
\end{align*}
Moreover, there are $a_h, r \in K[X]$ such that $\spoly(f,g) = \sum_{h \in F_{I_\tau}} a_h h + r$. Due to the
above computation it follows $$\spoly(\sigma_\tau(f),\sigma_\tau(g)) = \xi_k^{\nu_\alpha-\nu_\gamma} \cdot
\sigma_\tau \left( \sum_{h \in F_{I_\tau}} a_h h + r \right) = \sum_{h \in F_{I_\tau}} b_h h + \xi_k^{\nu_\alpha
-\nu_\gamma} \cdot \sigma_\tau(r),$$ for suitable $b_h=\xi_k^{\nu_\alpha-\nu_\gamma} \cdot 
a_{\sigma_\tau^{-1}(h)} \in K[X]$ since $I_\tau$ respectively $F_{I_\tau}$ is $\sigma_\tau$-symmetric and 
consequently $$\NF\big(\spoly(\sigma_\tau(f), \sigma_\tau(g)), F_{I_\tau} \big) = \xi_k^{\nu_\alpha-\nu_\gamma} 
\cdot \sigma_\tau \left( \NF\big(\spoly(f,g), F_{I_\tau}\big) \right).$$ This property implies that the reduced 
Gr\"obner basis $G_\tau = \{g^\tau_1,\ldots, g^\tau_s \}$ of $I_\tau$ satisfies $\sigma_\tau(g^\tau_i) = 
\xi_k^{\nu_{ij}} \cdot g^\tau_j$ for suitable $i,j \in \set 1s$ and $\nu_{ij} \in \left\{0,\ldots,k-1\right\}$. Moreover, 
it follows $\LM(g^\tau_i) = \LM(\sigma_\tau(g^\tau_i)) = \LM(g^\tau_j)$, but since $G_\tau$ is reduced we 
conclude $g^\tau_i = g^\tau_j$. Hence, we have $\sigma_\tau(g^\tau_i) = \xi_k^{\nu_{i}} \cdot g^\tau_i$ for all 
$i \in \set 1s$ with suitable $\nu_{i} \in \set{0}{k-1}$. 
\end{proof}

\begin{exmp} \label{exmpMain3}
Let $I_\tau =  \langle x^4-2x^2y^2+y^4-z,\, -x^2+y^2+2x-2y+3z,\, -x^2+y^2-2x-2y+3z \rangle \subseteq \Q[x,y,z]$
and $\sigma_\tau \in \Aut(\Q[x,y,z])$ as obtained in Example \ref{exmpMain2}. Then $I_\tau$ is $\sigma_\tau$-
symmetric, and its Gr\"obner basis $$G_\tau=\{x,\; 12yz-9z^2-8y+13z,\; y^2-2y+3z,\; 81z^3+36z^2-56y+115z\}$$ 
satisfies $\sigma_\tau(g_\tau) = (-1)^{\nu_{g_\tau}} \cdot g_\tau$ for suitable $\nu_{g_\tau} \in \{1,2\}$ and all 
$g_\tau \in G_\tau$. Now, the reverse transformation of $G_\tau$ yields the set 
\begin{align*}
\tau^{-1}(G_\tau) = \big\{ & \tfrac 12x-\tfrac 12y,\; 6xz+6yz-9z^2-4x-4y+13z,\; \\
                                     & \tfrac 14x^2+\tfrac 12xy+\tfrac 14y^2-x-y+3z,\; 81z^3+36z^2-28x-28y+115z \big\}.
\end{align*}
\end{exmp}

Obviously, just pulling back the Gr\"obner basis $G_\tau$ via $\tau^{-1}$ does not lead to a Gr\"obner basis of
the input ideal $I$. Thus, we have to compute a Gr\"obner basis of the ideal $\gen{\tau^{-1}(G_\tau)}$ as well.
Nevertheless, the advantage of this computation is the fact that the achieved property as described in Proposition
\ref{propPropertyGB} is respected by applying $\tau^{-1}$ on $G_\tau$. More precisely, the following proposition 
holds.

\begin{prop} \label{propPropertyGB2}
$\sigma(g) = \xi_k^{\nu_g} \cdot g$ for all $g \in \tau^{-1}(G_\tau)$ and suitable $0 \leq \nu_g \leq k-1$.
\end{prop}

\begin{proof}
Let $g \in \tau^{-1}(G_\tau)$, i.e. there is an $g_\tau \in G_\tau$ such that $g = \tau^{-1}(g_\tau)$. Then 
due to Proposition \ref{propAut} and Proposition \ref{propPropertyGB} we have $$\tau(\sigma(g)) = 
\tau(\sigma(\tau^{-1}(g_\tau))) = \sigma_\tau(g_\tau) = \xi_k^{\nu_{g_\tau}} \cdot g_\tau$$ for some 
$\nu_{g_\tau} \in \set{0}{k-1}$.
Hence, we obtain $$\sigma(g) = \tau^{-1}(\tau(\sigma(g))) = \tau^{-1}\left(\xi_k^{\nu_{g_\tau}} \cdot g_\tau 
\right) = \xi_k^{\nu_{g_\tau}} \cdot \tau^{-1}(g_\tau) = \xi_k^{\nu_{g_\tau}} \cdot g.$$ 
This proves the proposition.
\end{proof}

\begin{exmp} \label{exmpMain4}
Let $\sigma = (12)(3) \in \sym_3$ with $\ord(\sigma) = 2$, $\xi_2 = -1 \in \Q$ and $\tau^{-1}(G_\tau)$ as 
obtained in Example \ref{exmpMain3}. Then we compute
\begin{align*}
\sigma(\tfrac 12x-\tfrac 12y) & = -(\tfrac 12x-\tfrac 12y), \\
\sigma(6xz+6yz-9z^2-4x-4y+13z) & = 6xz+6yz-9z^2-4x-4y+13z, \\
\sigma(\tfrac 14x^2+\tfrac 12xy+\tfrac 14y^2-x-y+3z) & = \tfrac 14x^2+\tfrac 12xy+\tfrac 14y^2-x-y+3z, \\
\sigma(81z^3+36z^2-28x-28y+115z) & = 81z^3+36z^2-28x-28y+115z,
\end{align*}
as claimed in Proposition \ref{propPropertyGB2}.
\end{exmp}

The following diagram summarizes and illustrates our way of improving the computation of a Gr\"obner 
basis $G$ of a $\sigma$-symmetric ideal $I$. 

$$\begin{CD}
\big(I,\sigma\big)         @>\tau>>          \big(I_\tau,\sigma_\tau\big) \\
@.                                           @VV\texttt{std}V \\
\gen{\tau^{-1}(G_\tau)}        @<\tau^{-1}<<     G_\tau \\
@VV\texttt{std}V                             @.  \\
G                          @.                @.
\end{CD}$$ 

\vspace{0.2cm}

Note that the linear transformation $\tau$ is defined in Remark \ref{remAut} and the procedure \texttt{std} is 
implemented in \textsc{Singular} and computes a Gr\"obner basis (standard basis) of the input. 

Algorithm \ref{algSymmStd} computes the Gr\"obner basis of a $\sigma$-symmetric ideal $I$.\footnote{The 
corresponding procedures are implemented in \singular in the library \texttt{symodstd.lib}.}

\begin{algorithm}[h] 
\caption{Symmetric Gr\"obner Basis Computation (\texttt{symmStd})} \label{algSymmStd}
\begin{flushleft} Assume that $>$ is a degree ordering. \end{flushleft}

\begin{algorithmic}[1]
\REQUIRE $I \subseteq K[X]$ and $\sigma \in \sym_n$, such that $I$ is $\sigma$-symmetric.
\ENSURE $G \subseteq K[X]$, the Gr\"obner basis of $I$.
\vspace{0.1cm}
\STATE $k=\ord(\sigma)$;
\IF{$k \mod \ch(K) = 0$}
\PRINT Warning, algorithm is not applicable.
\RETURN $\emptyset$;
\ENDIF
\IF{$k=2$ \;or\; $(\ch(K)-1) \mod k = 0$}
\STATE compute $\xi_k \in K$;
\ELSE
\STATE $K = K[a]/\Phi_k(a)$;
\STATE $\xi_k:=a$;  
\ENDIF
\STATE compute $\tau \in \Aut(K[X])$;
\STATE $G_\tau = \texttt{std}(I_\tau)$;
\STATE $G = \texttt{std}(\langle\tau^{-1}(G_\tau)\rangle)$;
\RETURN $G$;
\end{algorithmic}
\end{algorithm}

\begin{thm} \label{thmSymmStd}
Algorithm \ref{algSymmStd} terminates and is correct, i.e.\ the output $G$ is a Gr\"obner basis of the
input $I$.
\end{thm}

\begin{proof}
Termination is clear and for proving correctness it suffices to show that $I = \gen{\tau^{-1}(G_\tau)}$
since $G$ is by definition a Gr\"obner basis of $\gen{\tau^{-1}(G_\tau)}$. 
Let $f \in I$ and $G_\tau = \set{g^\tau_1}{g^\tau_s}$. Then $\tau(f) \in \tau(I) = I_\tau$ and consequently 
there are $a_1,\ldots,a_s \in K[X]$ such that $\tau(f)=\sum_{i=1}^s a_i \cdot g^\tau_i$ since $G_\tau$ is 
a Gr\"obner basis of $I_\tau$.
Hence, we obtain $$f = \tau^{-1}(\tau(f)) = \sum_{i=1}^s \tau^{-1}(a_i) \cdot \tau^{-1}(g^\tau_i) \in 
\gen{\tau^{-1}(G_\tau)}.$$ For the other inclusion let $g \in \gen{\tau^{-1}(G_\tau)}$.
It follows that $\tau(g) \in \gen{G_\tau} = I_\tau = \tau(I)$ and moreover $g \in I$ since $\tau$ is an 
automorphism.
\end{proof}

For illustration of Algorithm \ref{algSymmStd} we combine all previous examples.

\begin{exmp}
Again, let $I = \gen{x^2y^2 - z,\; xy - 2y + 3z,\; xy - 2x + 3z} \subseteq \Q[x,y,z]$ and $\sigma = (12)(3) \in \sym_3$. 
Referring to Examples \ref{exmpMain1}, \ref{exmpMain2}, \ref{exmpMain3} and \ref{exmpMain4} we already 
obtained
\begin{align*}
I_\tau := \tau(I) = \big\langle x^4-2x^2y^2+y^4-z,\, & -x^2+y^2+2x-2y+3z,  \\
                                              & -x^2+y^2-2x-2y+3z \big\rangle,
\end{align*}
and its Gr\"obner basis $$G_\tau=\{x,\; 12yz-9z^2-8y+13z,\; y^2-2y+3z,\; 81z^3+36z^2-56y+115z\}$$
with 
\begin{align*}
\tau^{-1}(G_\tau) = \big\{ & \tfrac 12x-\tfrac 12y,\; 6xz+6yz-9z^2-4x-4y+13z,\; \\
                                     & \tfrac 14x^2+\tfrac 12xy+\tfrac 14y^2-x-y+3z,\; 81z^3+36z^2-28x-28y+115z \big\}.
\end{align*}
Finally, we compute
\begin{align*}
G = \{ & x-y,\; 12yz-9z^2-8y+13z,\; y^2-2y+3z,\; 
                       81z^3+36z^2-56y+115z \},
\end{align*}
the Gr\"obner basis of $\gen{\tau^{-1}(G_\tau)}$ respectively $I$.
\end{exmp}

\subsection{Examples and timings} \label{subsecSymmGBExTime}

In this section we provide examples on which we time the new algorithm \texttt{symmStd} (cf.\ 
Algorithm \ref{algSymmStd}) as opposed to the algorithm \texttt{std} implemented in \textsc{Singular}
(cf.\ \cite{DGPS12}). 
Timings are conducted by using \singular{3-1-3} on an AMD Opteron 6174 machine with $48$ CPUs, 
$2.2$ GHz, and $128$ GB of RAM running the Gentoo Linux operating system.

A more detailed description of the considered examples can be found in Section 
\ref{subsecSymmGBQExTime}.

\begin{exmp} \label{exmpCyclic7}
Cyclic $7$-roots, $\sigma_1 = (16)(25)(34) \in \sym_7$ with $\ord(\sigma_1) = 2$, $\sigma_2 = 
(1234567) \in \sym_7$ with $\ord(\sigma_2) = 7$.
\begin{small}
\begin{table}[H]
\begin{center}
\begin{tabular}{|c||c|c|c|c|c|}
\hline
\backslashbox{Algorithm}{$\ch(K)$} & $127$ & $30817$ & $100003$ & $2147483647$ \\
\hline \hline
\texttt{std} [sec] & 2 & 2 & 3 & 11 \\ \hline \hline
$\texttt{symmStd}(\_\,,\sigma_1)$ [sec] & 1 & 2 & 2 & 7 \\ \hline
$\texttt{symmStd}(\_\,,\sigma_1)/\texttt{std}$ & 0.50 & 1.00 & 0.67 & 0.64 \\ \hline \hline
$\texttt{symmStd}(\_\,,\sigma_2)$ [sec] & 5 & 5 & 7 & 20 \\ \hline
$\texttt{symmStd}(\_\,,\sigma_2)/\texttt{std}$ &
\hspace{0.4cm} 2.50 \hspace{0.4cm} & \hspace{0.4cm} 2.50 \hspace{0.4cm} & 
\hspace{0.4cm} 2.33 \hspace{0.4cm} & \hspace{0.4cm} 1.82 \hspace{0.4cm} \\ \hline
\end{tabular}
\end{center}
\end{table}
\end{small}
\end{exmp}

\begin{rem} \label{remSymmStd}
Note that in case of Example \ref{exmpCyclic7} the pure Gr\"obner basis computation is comparably 
easy so that the symmetry based approach decelerates the whole computation when applying the 
permutation of order $7$ so that the usage of the linear transformation $\tau$ (cf.\ Remark \ref{remAut}) 
dominates the process. This circumstance will partially also be transpired in the following examples.

Consequently, a permutation of higher order usually accelerates the Gr\"obner basis computation
on the transformed side but, on the other hand, may also decelerate the whole algorithm because
of an expensive application of the linear transformation.

However, summing up, we achieve an enormous advancement via \texttt{symmStd} (see the
following examples) although we have to compute Gr\"obner bases internally twice on modified 
input ideals via \texttt{std}.
\end{rem}

\vspace{0.1cm}

\begin{exmp}
Cyclic $8$-roots, $\sigma_1 = (18)(27)(36)(45) \in \sym_8$ with $\ord(\sigma_1) = 2$, $\sigma_2 = 
(1753)(2864) \in \sym_8$ with $\ord(\sigma_2) = 4$, $\sigma_3 = (12345678) \in \sym_8$ with 
$\ord(\sigma_3) = 8$.
\begin{small}
\begin{table}[H]
\begin{center}
\begin{tabular}{|c||c|c|c|c|c|}
\hline
\backslashbox{Algorithm}{$\ch(K)$} & $137$ & $30817$ & $100049$ & $2147483497$ \\
\hline \hline
\texttt{std} [sec] & 69 & 79 & 104 & 125 \\ \hline \hline
$\texttt{symmStd}(\_\,,\sigma_1)$ [sec] & 49 & 59 & 76 & 93 \\ \hline
$\texttt{symmStd}(\_\,,\sigma_1)/\texttt{std}$ & 0.71 & 0.75 & 0.73 & 0.74 \\ \hline \hline
$\texttt{symmStd}(\_\,,\sigma_2)$ [sec] & 32 & 36 & 46 & 55 \\ \hline
$\texttt{symmStd}(\_\,,\sigma_2)/\texttt{std}$ & 0.46 & 0.46 & 0.44 & 0.44 \\ \hline \hline
$\texttt{symmStd}(\_\,,\sigma_3)$ [sec] & 54 & 57 & 70 & 82 \\ \hline
$\texttt{symmStd}(\_\,,\sigma_3)/\texttt{std}$ &
\hspace{0.4cm} 0.78 \hspace{0.4cm} & \hspace{0.4cm} 0.72 \hspace{0.4cm} & 
\hspace{0.4cm} 0.67 \hspace{0.4cm} & \hspace{0.4cm} 0.66 \hspace{0.4cm} \\ \hline
\end{tabular}
\end{center}
\end{table}
\end{small}
\end{exmp}

\vspace{0.2cm}

\begin{exmp}
Cyclic $9$-roots, $\sigma_1 = (18)(27)(36)(45) \in \sym_9$ with $\ord(\sigma_1) = 2$, $\sigma_2 = 
(147)(258)(369) \in \sym_9$ with $\ord(\sigma_2) = 3$, $\sigma_3 = (123456789) \in \sym_9$ with 
$\ord(\sigma_3) = 9$.
\begin{small}
\begin{table}[H]
\begin{center}
\begin{tabular}{|c||c|c|c|c|c|}
\hline
\backslashbox{Algorithm}{$\ch(K)$} & $181$ & $30817$ & $100153$ & $2147483647$ \\
\hline \hline
\texttt{std} [sec] & 16458 & 17312 & 21077 & 24697 \\ \hline \hline
$\texttt{symmStd}(\_\,,\sigma_1)$ [sec] & 10655 & 9955 & 10881 & 13077 \\ \hline
$\texttt{symmStd}(\_\,,\sigma_1)/\texttt{std}$ & 0.65 & 0.58 & 0.52 & 0.53 \\ \hline \hline
$\texttt{symmStd}(\_\,,\sigma_2)$ [sec] &  002 & 4554 & 5471 & 6419 \\ \hline
$\texttt{symmStd}(\_\,,\sigma_2)/\texttt{std}$ & 0.24 & 0.26 & 0.26 & 0.26 \\ \hline \hline
$\texttt{symmStd}(\_\,,\sigma_3)$ [sec] & 4016 & 3756 & 4464 & 5272 \\ \hline
$\texttt{symmStd}(\_\,,\sigma_3)/\texttt{std}$ &
\hspace{0.4cm} 0.24 \hspace{0.4cm} & \hspace{0.4cm} 0.22 \hspace{0.4cm} & 
\hspace{0.4cm} 0.21 \hspace{0.4cm} & \hspace{0.4cm} 0.21 \hspace{0.4cm} \\ \hline
\end{tabular}
\end{center}
\end{table}
\end{small}
\end{exmp}

\vspace{0.2cm}

\begin{exmp}
100 Swiss Francs Problem, $\sigma = (45)(89) \in \sym_9$ with $\ord(\sigma) = 2$.
\begin{small}
\begin{table}[H]
\begin{center}
\begin{tabular}{|c||c|c|c|c|c|}
\hline
\backslashbox{Algorithm}{$\ch(K)$} & $181$ & $30817$ & $100153$ & $2147483647$ \\
\hline \hline
\texttt{std} [sec] & 5 & 5 & 6 & 8 \\ \hline
$\texttt{symmStd}(\_\,,\sigma)$ [sec] & 2 & 3 & 4 & 5 \\ \hline
$\texttt{symmStd}(\_\,,\sigma)/\texttt{std}$ &
\hspace{0.4cm} 0.40 \hspace{0.4cm} & \hspace{0.4cm} 0.60 \hspace{0.4cm} & 
\hspace{0.4cm} 0.67 \hspace{0.4cm} & \hspace{0.4cm} 0.63 \hspace{0.4cm} \\ \hline
\end{tabular}
\end{center}
\end{table}
\end{small}
\end{exmp}

\begin{exmp}
$7 - 4.3^2 - 4.3^2$ for $\sym_{11}$, $\sigma = (15)(26)(37)(48) \in \sym_{10}$ with 
$\ord(\sigma) = 2$.
\begin{small}
\begin{table}[H]
\begin{center}
\begin{tabular}{|c||c|c|c|c|c|}
\hline
\backslashbox{Algorithm}{$\ch(K)$} & $181$ & $30817$ & $100153$ & $2147483647$ \\
\hline \hline
\texttt{std} [sec] & 3 & 4 & 5 & 6 \\ \hline
$\texttt{symmStd}(\_\,,\sigma)$ [sec] & 1 & 2 & 2 & 2 \\ \hline
$\texttt{symmStd}(\_\,,\sigma)/\texttt{std}$ &
\hspace{0.4cm} 0.33 \hspace{0.4cm} & \hspace{0.4cm} 0.50 \hspace{0.4cm} & 
\hspace{0.4cm} 0.40 \hspace{0.4cm} & \hspace{0.4cm} 0.33 \hspace{0.4cm} \\ \hline
\end{tabular}
\end{center}
\end{table}
\end{small}
\end{exmp}

\begin{exmp}
$7 - 5.4 - 5.4$ for $\sym_{11}$, $\sigma = (15)(26)(37)(48) \in \sym_{10}$ with 
$\ord(\sigma) = 2$.
\begin{small}
\begin{table}[H]
\begin{center}
\begin{tabular}{|c||c|c|c|c|c|}
\hline
\backslashbox{Algorithm}{$\ch(K)$} & $181$ & $30817$ & $100153$ & $2147483647$ \\
\hline \hline
\texttt{std} [sec] & 4 & 4 & 5 & 6 \\ \hline
$\texttt{symmStd}(\_\,,\sigma)$ [sec] & 1 & 1 & 2 & 3 \\ \hline
$\texttt{symmStd}(\_\,,\sigma)/\texttt{std}$ &
\hspace{0.4cm} 0.25 \hspace{0.4cm} & \hspace{0.4cm} 0.25 \hspace{0.4cm} & 
\hspace{0.4cm} 0.40 \hspace{0.4cm} & \hspace{0.4cm} 0.50 \hspace{0.4cm} \\ \hline
\end{tabular}
\end{center}
\end{table}
\end{small}
\end{exmp}

\section{Gr\"obner bases using symmetry and modular methods} \label{secSymmGBQ}

When applying Algorithm \ref{algSymmStd} on $\sigma$-symmetric ideals over the rationals so that 
$\ord(\sigma) = k >2$ we need to swap to $\Q[a]/\Phi_k(a)$ as explained in Remark \ref{remFieldExtension}.
However, over fields $K$ of positive characteristic such that $k \mid (\ch(K) - 1)$ this can be omitted.
Consequently, we use modular methods to improve Algorithm \ref{algSymmStd} applied on 
$\sigma$-symmetric ideals in the polynomial ring over the rationals. More precisely, we improve the 
modular Gr\"obner basis algorithm as introduced by Arnold (cf.\ \cite{A03}), Idrees, Pfister, Steidel 
(cf.\ \cite{IPS11}) and Noro, Yokoyama (cf.\ \cite{NY12}, \cite{Y12}).

\begin{rem} \label{remAddVerif}
Noro and Yokoyama revealed that \cite[Theorem 2.4]{IPS11} for the inhomogeneous case is
only correct with an additional assumption. 

Let $I \subseteq \Q[X]$ be an ideal generated by a finite subset $F_I$. For homogenization we provide
an extra variable $t$ and define $f^h := t^d \cdot f(x_1/t,\ldots,x_n/t)$ for $f \in \Q[X]$ where $d$ is
the total degree of $f$ and $F_I^h := \{f^h \mid f \in F_I\}$. Moreover, $>$ induces a monomial ordering
$>_h$ such that $X^\alpha t^a >_h X^\beta t^b$ if and only if either $|\alpha| + a > |\beta| + b$ or
$(|\alpha| + a = |\beta| + b$ and $X^\alpha > X^\beta)$.
Then Noro and Yokoyama advise to add the condition that $p$ is lucky for $\langle F_I^h \cup \{t^m\}
\rangle$ with respect to $>_h$ where $m$ is an integer such that $\big(\langle \Phi_p(F_I^h)\rangle : 
t^m\big) = \big(\langle \Phi_p(F_I^h)\rangle : t^\infty\big)$ and $\Phi_p$ denotes the canonical
projection to $\F_p[X]$ for a prime $p$ (cf.\ \cite{NY12}, \cite{Y12}).
\end{rem}

\subsection{The probabilistic symmetric modular Gr\"obner basis algorithm} \label{subsecSymmGBQAlg}

Algorithm \ref{algSyModStd} combines the algorithms \texttt{symmStd} (cf.\ Algorithm \ref{algSymmStd}) 
and a modification of \texttt{modStd} (cf.\ \cite[Algorithm 1]{IPS11}).\footnote{The corresponding procedures 
are implemented in \singular in the library \texttt{symodstd.lib}.}

\begin{algorithm}[h] 
\caption{Symmetric Modular Gr\"obner Basis Computation (\texttt{syModStd})} \label{algSyModStd}
\begin{flushleft} Assume that $>$ is a degree ordering. \end{flushleft}

\begin{algorithmic}[1]
\REQUIRE $I \subseteq \Q[X]$ and $\sigma \in \sym_n$, such that $I$ is $\sigma$-symmetric.
\ENSURE $G \subseteq \Q[X]$, the Gr\"obner basis of $I$.
\vspace{0.1cm}
\STATE $k=\ord(\sigma)$;
\STATE choose $P$, a list of random primes such that $k \mid (p-1)$ for all $p \in P$;
\STATE $GP = \emptyset$;
\LOOP
\FOR{$p \in P$}
\STATE $G_p = \texttt{symmStd}(I_p, \sigma)$;
\STATE $GP = GP \cup \{G_{p}\}$;
\ENDFOR
\STATE $(GP,P) = \texttt{deleteUnluckyPrimesSB}(GP,P)$;
\STATE lift $(GP,P)$ to $G \subseteq \Q[X]$ by applying Chinese remainder algorithm and Farey rational map;
\IF{$G$ passes \texttt{finalVerificationTests}}
\RETURN $G$;
\ENDIF
\STATE enlarge $P$;
\ENDLOOP
\end{algorithmic}
\end{algorithm}

\begin{rem}
The essential differences of the algorithm \texttt{symModStd} compared to the algorithm \texttt{modStd} 
are the following:
\begin{enumerate}
\item The choice of the prime list $P$ has to be restricted. Every considered prime number $p \in P$ has to
          satisfy the condition $k \mid (p-1)$ in order to assure that the coefficient field $\F_p$ has a $k$--th 
          primitive root of unity.
\item The modular Gr\"obner bases $G_p$ are computed via \texttt{symmStd} instead of \texttt{std}.
\end{enumerate}
Similar to \texttt{modStd} we can parallelize Algorithm \ref{algSyModStd} by computing the modular 
Gr\"obner bases $G_p$ respectively performing the final tests in parallel.
\end{rem}

\begin{thm} \label{thmSyModStd}
Algorithm \ref{algSyModStd} terminates and is correct, i.e.\ the output $G$ is a Gr\"obner basis of the
input $I$.
\end{thm}

\begin{proof}
Termination is clear and correctness follows directly from Theorem \ref{thmSymmStd} and 
the improvement of \cite[Theorem 2.4]{IPS11} by Noro and Yokoyama (cf.\ \cite{NY12}, \cite{Y12}).
\end{proof}

The symmetric part of the symmetric modular Gr\"obner basis algorithm is not influenced by the additional
verification test mentioned in Remark \ref{remAddVerif} and for this verification part the ideal $\langle F_I^h 
\cup \{t^m\} \rangle$ is homogeneous, and symmetric if $I$ is so, such that Algorithm \ref{algSyModStd}
can be directly applied to it without verifying the additional condition (cf.\ \cite{A03}).

Nevertheless, the additional verification test, in general, decelerates the whole algorithm considerably.
In order to emphasize the impact of Algorithm \ref{algSymmStd} we therefore just time a probabilistic
variant (call it $\texttt{syModStd}^*$ and $\texttt{modStd}^*$, respectively) by skipping the additional
verification due to Noro and Yokoyama.

\subsection{Examples and timings} \label{subsecSymmGBQExTime}

In this section we provide examples on which we time the new algorithms \texttt{symmStd} (cf.\ 
Algorithm \ref{algSymmStd}) respectively $\texttt{syModStd}^*$ (cf.\ Algorithm \ref{algSyModStd}) 
and its parallelization as opposed to the former algorithms \texttt{std} respectively 
$\texttt{modStd}^*$ implemented in \textsc{Singular}  (cf.\ \cite{DGPS12}). 
Again, all timings are conducted by using \singular{3-1-3} on an AMD Opteron 6174 machine with 
$48$ CPUs, $2.2$ GHz, and $128$ GB of RAM running the Gentoo Linux operating system.

\begin{exmp}[Cyclic $n$-roots (cf.\ \cite{Bj85}, \cite{Bj90}, \cite{BF91})] \label{exmpCyclic}
The task to compute a Gr\"obner basis of the ideal in $\Q[X]=\Q[x_1,\ldots,x_n]$ corresponding to the 
following system of polynomial equations 
\begin{align*}
x_1+\ldots+x_n & = 0, \\
x_1x_2 + x_2x_3 + \ldots + x_{n-1}x_n + x_nx_1 & = 0, \\
\vdots & \quad\;\, \vdots \\
x_1x_2\cdots x_{n-1} + x_2x_3 \cdots x_n + \ldots + x_{n-1}x_n\cdots x_{n-3} + x_n x_1 \cdots x_{n-2} & = 0, \\
x_1\cdots x_n - 1 & = 0
\end{align*}
has become a benchmark problem for Gr\"obner basis techniques. We call this ideal \emph{cyclic$(n)$},
and its variety \emph{cyclic $n$-roots} (cf.\ \cite{Bj85}). The origin of the problem is related to Fourier
analysis (cf.\ \cite{Bj85}, \cite{Bj90}). 
Obviously, the ideal cyclic$(n)$ is by definition symmetric with respect to the $n$-cycle $\sigma_n = (1\ldots n)$
such that we can apply the algorithms \texttt{symmStd} and $\texttt{syModStd}^*$.

Until the end of 2009, \singular was able to compute a Gr\"obner basis of cyclic$(n)$ for $n \leq 8$. In
April 2010, we could for the first time compute a Gr\"obner basis of cyclic$(9)$ via a prototype of $\texttt{syModStd}^*$
by using the 32-bit version of \singular{3-1-1} on an Intel\textregistered \ Xeon\textregistered \ X5460 machine with 
$4$ CPUs, 3.16 GHz each, and 64 GB of RAM under the Gentoo Linux operating system within 23 days. 

Table \ref{tabCyclic} summarizes the present timings for computing a Gr\"obner basis of cyclic$(n)$ for $n=7,8,9$ 
with different numbers of cores $\ell$ denoted by $\texttt{modStd}^*(\ell)$ respectively $\texttt{syModStd}^*(\ell)$, 
and different permutations $\sigma$ where again $k = \ord(\sigma)$ denotes the order of $\sigma$.

\begin{small}
\begin{table}[hbt]
\begin{center}
\begin{tabular}{|r|r|r|r|r|r|r|}
\hline
$n$ & $k$ & \multicolumn{1}{c|}{\texttt{symmStd}} & \multicolumn{1}{c|}{$\texttt{modStd}^*(1)$} 
& \multicolumn{1}{c|}{$\texttt{modStd}^*(30)$} & \multicolumn{1}{c|}{$\texttt{syModStd}^*(1)$}
& \multicolumn{1}{c|}{$\texttt{syModStd}^*(30)$} \\
\hline \hline
$7$ & $2$ & 317 & 111 & 34 & 69 & 29 \\ \hline
$7$ & $7$ & \hspace{0.7cm} 409394 & \hspace{1.3cm} 106 & \hspace{1.4cm} 38 & \hspace{1.3cm} 165 
& \hspace{1.4cm} 50 \\ \hline
$8$ & $2$ & - & 6816 & 973 & 5196 & 811 \\ \hline
$8$ & $4$ & - & 6816 & 973 & 3120 & 620 \\ \hline
$8$ & $8$ & - & 6816 & 973 & 4454 & 788 \\ \hline
$9$ & $3$ & - & 9935103 & 475981 & 2790303 & 207681 \\ \hline 
\end{tabular}
\end{center}
\hspace{15mm}
\caption{Total running times in seconds for computing a Gr\"obner basis of cyclic$(n)$ for $n=7,8,9$
via \texttt{symmStd}, $\texttt{modStd}^*(\ell)$, and $\texttt{syModStd}^*(\ell)$ for $\ell = 1,30$, 
using different permutations of order $k$. The symbol ``-'' indicates out of memory failures.} 
\label{tabCyclic}
\end{table}
\end{small}
In these examples we used the permutations $(16)(25)(34), (1234567) \in \sym_7$ for $n = 7$, the permutations 
$(18)(27)(36)(45), (1753)(2864), (12345678) \in \sym_8$ for $n=8$, and the permutation $(147)(258)(369) 
\in \sym_9$ for $n=9$.

Note that the symmetric approach in the modular version just influences the Gr\"obner basis computation
in each prime characteristic. This means, on the one hand, that in most cases the use of parallelization
decreasing the number of sequentially accomplished Gr\"obner basis computations is more decisive for 
higher efficiency than just applying the symmetry based approach.
On the other hand, if the calculations in positive characteristic are comparably easy as in the 
cyclic$(7)$-case, then the symmetry based approach may even slow down the whole process since the 
usage of the linear transformation $\tau$ and the second Gr\"obner basis computation in each prime 
characteristic overrun the original calculations of the purely modular approach.

Moreover, note that the timings obtained by the modular versions are dependent on the used permutation 
and especially on its order. In particular, a higher order $k$, that is a higher symmetry, speeds up the 
Gr\"obner basis computation on the transformed side but in contrast slows down the application of the linear 
transformation $\tau$ (cf.\ Remark \ref{remAut}) and, in addition, allocates more memory since the support 
of a ring variable's image depends on the order $k$ of the permutation.
This circumstance justifies that applying the symmetric modular algorithm for computing a Gr\"obner basis 
of cyclic$(8)$ is most performant when using the permutation $(1753)(2864) \in \sym_8$ of order $4$.
Similarly, we make use of the permutation $(147)(258)(369) \in \sym_9$ of order $3$ for cyclic$(9)$ since
considering a permutation of order $9$ implies substituting each ring variable by a linear combination of
$9$ ring variables when applying the linear transformation $\tau$, so that the parallel computation crashes
because of memory overflow.
\end{exmp}

\begin{exmp}[$100$ Swiss Francs Problem (cf.\ \cite{ZJG11}, \cite{St08})]
Sturmfels offered a cash prize of $100$ Swiss Francs for the resolution of a very specific conjecture
in the \emph{Nachdiplomsvorlesung} (postgraduate course) which he held at ETH Z\"urich in the summer
of 2005. Based on a concrete biological example proposed in \cite[Example 1.16]{PS05} the problem
arisen to maximize the likelihood function $$L(P) = \left(\prod_{i=1}^4 p_{ii} \right)^4 \cdot \left( \prod_{i \neq j}
p_{ij} \right)^2 \cdot \left( \sum_{i,j = 1}^4 p_{ij}\right)^{-40}$$ over all (positive) $4 \times 4$-matrices $P =
(p_{ij})_{1 \leq i,j \leq 4}$ of rank at most two. Due to numerical experiments by applying an 
expectation-maximization algorithm (EM algorithm), B.\ Sturmfels conjectured that the matrix $$P =
\frac{1}{40} \matrix{cccc}{3&3&2&2\\ 3&3&2&2\\ 2&2&3&3\\ 2&2&3&3}$$ is a global maximum of the
likelihood function $L(P)$ (cf.\ \cite{St08}).

The conjecture is positively confirmed in \cite{ZJG11}. In their approach via Gr\"obner bases 
(cf.\ \cite[Section 2.3]{ZJG11}) it is necessary to compute the Gr\"obner basis of the ideal $J$ defined by
\begin{align*}
I & = \gen{a_1-b_1, \sum_{i=1}^4a_1, \sum_{i=1}^4b_i,f_1,\ldots,f_4,g_1,\ldots,g_4} \subseteq 
         \Q[a_1,\ldots,a_4,b_1,\ldots,b_4]
\end{align*}
with
\begin{align*}
f_i & = \sum_{j=1}^4 \left(b_j \cdot (1+a_ib_i)\cdot\prod_{k\neq j}(1+a_ib_k) \right) + b_i \cdot 
             \prod_{k=1}^4(1+a_ib_k), \\
g_i & = \sum_{j=1}^4 \left(a_j \cdot (1+a_ib_i)\cdot\prod_{k\neq j}(1+a_kb_i) \right) + a_i \cdot 
             \prod_{k=1}^4(1+a_kb_i)
\end{align*}
for $1 \leq i \leq 4$, and
\begin{align*}
J = I + \gen{1-ua_1} \subseteq \Q[a_1,\ldots,a_4,b_1,\ldots,b_4,u]
\end{align*}
with respect to an elimination ordering on the variable $u$. 
In a first approach we therefore applied \texttt{modStd} using the lexicographical ordering $>_{lp}$ respectively the 
block ordering $(>_{dp(8)},>_{lp(1)})$ to eliminate the variable $u$. It turned out that both variants are comparably 
slow so that we used in a second approach the degree reverse lexicographical ordering $>_{dp}$, and applied the 
FGLM-algorithm (cf.\ \cite{FGLM93}) subsequently to obtain a Gr\"obner basis with respect to the block ordering 
$(>_{dp(8)},>_{lp(1)})$. Since the ideal $J \subseteq \Q[a_1,\ldots,a_4,b_1,\ldots,b_4,u]$ is symmetric with respect 
to the permutation $(34)(78) \in \sym_{9}$ we could moreover apply the algorithm $\texttt{syModStd}^*$. The timings 
for the computations in \singular are summarized in Table \ref{tab100SwissFrancs}.

\begin{small}
\begin{table}[hbt]
\begin{center}
\begin{tabular}{|c|c|}
\hline
Method & Running Time \\
\hline \hline
$\texttt{modStd}^*[>_{lp}]$ & 39919 \\ \hline
$\texttt{modStd}^*[(>_{dp(8)},>_{lp(1)})]$ & 515 \\ \hline
$\texttt{syModStd}^*[(>_{dp(8)},>_{lp(1)})]$ & 356 \\ \hline
$\texttt{modStd}^*[>_{dp}]$ -- \texttt{fglm}$[(>_{dp(8)},>_{lp(1)})]$ & 375 \\ \hline
$\texttt{syModStd}^*[>_{dp}]$ -- \texttt{fglm}$[(>_{dp(8)},>_{lp(1)})]$ & 284 \\ \hline
\end{tabular}
\end{center}
\hspace{15mm}
\caption{Total running times in seconds for computing the Gr\"obner basis of $J \subseteq \Q
[a_1,\ldots,a_4,b_1,\ldots,b_4,u]$ with respect to an elimination ordering on the variable $u$
via different methods.} 
\label{tab100SwissFrancs}
\end{table}
\end{small}
\end{exmp}

\begin{exmp}[Inverse Galois Problem (cf.\ \cite{Mal94}, \cite{Mat87}, \cite{MM99})]
A major topic in algebraic number theory is the inverse Galois problem over a field $K$, i.e.\ the 
question whether any finite group $G$ is the Galois group of a Galois extension of $K$.
The most interesting case is $K = \Q$ which is still open in general.
In contrast, the problem is known to be true for $K$ being a rational function field in one variable $t$ 
over an algebraically closed field of characteristic zero.
In particular, it is true for $K=\C(t)$, and in this case it is solved via geometric field extensions (see for
example \cite[I, \S1]{MM99}). Moreover, the same strategy applies to finite field extensions
of $\overline \Q(t)$ ramified only over $\{0,1,\infty\}$ (see for example \cite[I, \S5]{MM99}).
In this situation, for any triple $\sigma=(\sigma_1,\sigma_2,\sigma_3)$ of elements generating a transitive 
subgroup $G=\gen{\sigma_1,\sigma_2,\sigma_3} \subseteq \sym_n$ with $\sigma_1\sigma_2\sigma_3 = 1$ 
there exists a certain field extension $\overline K_\sigma/\overline{\Q}(t)$ of degree $n$, unramified outside 
$\{0,1,\infty\}$, and whose Galois group is isomorphic to $G$.
In fact, any such extension $\overline K_\sigma/\overline \Q(t)$ is already defined over a number field 
$k_\sigma = \Q(\alpha_\sigma)$, the so-called field of definition of $\overline K_\sigma/\overline \Q(t)$, 
so that there exists a further field extension $K_\sigma/k_\sigma(t)$ which also has $G$ as its Galois group.
The degree $[k_\sigma:\Q]$ is bounded from above by group theoretical information (see for example 
\cite[Proposition A]{Mal94}).
In order to construct the extension $\overline K_\sigma/\overline \Q(t)$ it is necessary to solve a system of
polynomial equations (see \cite[I, \S9]{MM99}). In case that, for example, $\sigma_1$ and $\sigma_2$ have 
the same cycle type, the defining ideal is symmetric with respect to a permutation 
of order $2$ so that we can apply the algorithms \texttt{symmStd} and \texttt{syModStd} to compute a Gr\"obner
basis of this system. In addition, choosing an elimination ordering for the last ring variable, the last 
polynomial $f$ of the Gr\"obner basis of this system of non-linear equations generates 
the field of definition $k_\sigma = \Q(\alpha_\sigma)$ insofar that $\alpha_\sigma$ is a zero of $f$. 
The irreducible factors of $f$ together with group theoretical information yield restrictions on $[k_\sigma:\Q]$.
In case that $[k_\sigma:\Q]=1$, the given group $G$ can even be realized over $\Q(t)$, and therefore also over 
$\Q$ by Hilbert's irreducibility theorem.

In $1994$, Malle collected computational data on several $k_\sigma$ of degree $[k_\sigma:\Q] \leq 13$ 
(cf.\ \cite{Mal94}) with the intention to observe regularities and hints to decrease the group theoretical bound. 
Table \ref{tabFieldExtension} lists further examples in the spirit of this article and which could not be computed 
at that time.

\begin{small}
\begin{table}[hbt]
\begin{center}
\begin{tabular}{|c|c|c|r|r|r|r|}
\hline
$n$ & \multicolumn{1}{|c|}{$G$} & \multicolumn{1}{c|}{$C_\sigma$} & \multicolumn{1}{c|}{$[k_\sigma : \Q]$} 
& \multicolumn{1}{c|}{\texttt{symmStd}} & \multicolumn{1}{c|}{$\texttt{modStd}^*$}
& \multicolumn{1}{c|}{$\texttt{syModStd}^*$} \\
\hline \hline
$7$ & $\A_7$ & $4.2 - 4.2 - 4.2$ & $12$ & \hspace{1.2cm} 24 & \hspace{1.2cm} 19 & \hspace{1.2cm} 19 \\ \hline
$9$ & $\sym_9$ & $4.2^2 - 4.2^2 - 5.3$ & $34$ & 5 & 15 & 13 \\ \hline
$10$ & $\A_{10}$ & $5.2^2 - 5.2^2 - 7$ & $37$ & 977 & 186 & 107 \\ \hline 
$11$ & $\A_{11}$ & $4.2 - 9 - 9$ & $12$ & 17 & 2 & 1 \\ \hline
$11$ & $\A_{11}$ & $4.2 - 4^2.3 - 4^2.3$ & $8$ & 34 & 3 & 2 \\ \hline 
$11$ & $\A_{11}$ & $4.2 - 5.3^2 - 5.3^2$ & $8$ & 15 & 3 & 2 \\ \hline
$11$ & $\A_{11}$ & $5 - 8.2 - 8.2$ & $11$ & 265 & 17 & 8 \\ \hline
$11$ & $\A_{11}$ & $5 - 6.4 - 6.4$ & $11$ & 339 & 20 & 10 \\ \hline
$11$ & $\A_{11}$ & $5 - 7.3 - 7.3$ & $11$ & 292 & 16 & 8 \\ \hline 
$11$ & $\sym_{11}$ & $7 - 4.3^2 - 4.3^2$ & $26$ & 631245 & 2506 & 1493 \\ \hline 
$11$ & $\sym_{11}$ & $7 - 4.3.2^2 - 4.3.2^2$ & $29$ & - & 3039 & 1979 \\ \hline
$11$ & $\sym_{11}$ & $7 - 7.2 - 7.2$ & $29$ & - & 1414 & 702 \\ \hline
$11$ & $\sym_{11}$ & $7 - 5.4 - 5.4$ & $26$ & - & 1738 & 899 \\ \hline
\end{tabular}
\end{center}
\hspace{15mm}
\caption{Total running times in seconds for computing the defining Gr\"obner basis of the field extension
$\overline K_\sigma/\overline{\Q}(t)$ having group $G$ and conjugacy class triple $C_\sigma$ (cf.\ \cite{Mal94}) 
via \texttt{symmStd}, $\texttt{modStd}^*$, and $\texttt{syModStd}^*$. Here, $C_\sigma$ is a class of $G$ 
containing elements of the given cycle type. The symbol ``-'' indicates out of memory failures.} 
\label{tabFieldExtension}
\end{table}
\end{small}

Note that all ideals belonging to the examples listed in Table \ref{tabFieldExtension} are zero-dimensio\-nal 
such that we can compute a Gr\"obner basis with respect to the degree reverse lexicographical ordering, and 
obtain a lexicographical Gr\"obner basis by applying the FGLM-algorithm (cf.\ \cite{FGLM93}) subsequently.
\end{exmp}

\section{Conclusion}

In all considered examples the symmetric (or equivalently the probabilistic symmetric modular) version 
of the Gr\"obner basis algorithm is the most performant one, and is, hence, a quite powerful tool if the 
input ideal is symmetric with respect to some permutation of the ring variables.

Although plenty of adaptive ideals are even symmetric under a whole permutation group the symmetry 
based approach presented in this article is designed for only a single permutation respectively cyclic 
subgroup of $\sym_n$. 
As already mentioned in Remark \ref{remSymmStd} and at the end of Example \ref{exmpCyclic} there 
is a particular conflict with respect to performance in the symmetric Gr\"obner basis algorithm between 
the Gr\"obner basis computations and the application of the linear transformation. 
Hence, a reasonable heuristic is to choose the applicable permutation $\sigma$ of cycle type $(l_1, \ldots,
l_{\vartheta(\sigma)})$ having maximal order $\lcm(l_1, \ldots, l_{\vartheta(\sigma)})$ and minimal 
$l_{\vartheta(\sigma)}$.

\section{Acknowledgement}

The author would like to thank his advisors Gunter Malle and Gerhard Pfister for introducing to the topic
of this article and constant support. In addition, he thanks Michael Cuntz, Christian Eder and Ulrich Thiel 
for helpful discussions. Finally, the author would like to thank Kazuhiro Yokoyama and the anonymous
referees whose comments and suggestions led to great improvement of the paper.

\end{document}